\documentclass[12pt]{article}
\usepackage{t1enc}
\usepackage[latin1]{inputenc}
\usepackage[english]{babel}
\usepackage{amsmath,amsthm,amssymb}
\usepackage{amsfonts}
\usepackage{latexsym}
\usepackage[dvips]{graphicx}
\usepackage{graphicx}
\usepackage[natural]{xcolor}
\usepackage{float}
\usepackage{enumerate}
\usepackage[subnum]{cases}
\usepackage{multirow}
\usepackage{hyperref}
\newcommand{\pf}{\noindent\textbf{Proof}.\quad}
\usepackage{listings}
\newtheorem{definition}{Definition}[section]
\newtheorem{theorem}[definition]{Theorem}

\newtheorem{lemma}[definition]{Lemma}
\newtheorem{claim}[definition]{Claim}

\newtheorem{corollary}[definition]{Corollary}

\usepackage{lineno}

\title{\bf Improved bounds on the Ramsey number of fans}
\author{Guantao Chen$^a$\thanks{Email: gchen@gsu.edu, partially supported by National Science Foundation grant DMS-1855716.}, \quad Xiaowei Yu$^{b}$\thanks{Email: xwyu@jsnu.edu.cn, partially supported  by the National Natural Science Foundation of China grants 11901252, 11871311, 12031018 and
the Scientific Research Foundation of Jiangsu Province grant 19KJB110010.},\quad Yi Zhao$^{a}$\thanks{Email: yzhao6@gsu.edu, partially supported by NSF grant DMS-1700622 and Simons Collaboration Grant 710094.}
\unskip\\[.5em]
{\small  $^a$ Department of Mathematics and Statistics, Georgia State University,}\\
{\small  Atlanta, GA 30303, USA}\\
{\small  $^b$ Department of Mathematics and Statistics, Jiangsu Normal University,}\\
{\small  Xuzhou,  221116, P. R. China}\\
}

\date{}
\begin{document}

\maketitle

\begin{abstract}
For a given graph $H$, the Ramsey number $r(H)$ is the minimum $N$ such that any 2-edge-coloring of the complete graph $K_N$ yields a
monochromatic copy of $H$. Given a positive integer $n$,  a \emph{fan }$F_n$ is a graph formed by $n$ triangles that share one common vertex. We show that ${9n}/{2}-5\le r(F_n)\le {11n}/{2} + 6$  for any $n$. This improves previous best bounds $r(F_n) \le 6n$ of Lin and Li and $r(F_n) \ge 4n+2$ of Zhang, Broersma and Chen.
\end{abstract}

\bigskip

\noindent {\bf Keywords:} Ramsey numbers; fans; books.

\section{Introduction}

Let $H_1$ and $H_2$ be two graphs. The \emph{Ramsey number} $r(H_1,H_2)$ is the minimum $N$ such that any red-blue coloring of the edges of the complete graph $K_N$ yields a red copy of $H_1$ or a blue copy of $H_2$.
Let $r(H) = r(H, H)$ be the diagonal Ramsey number.
Graph Ramsey theory is a central topic in graph theory and combinatorics.
For related results, see surveys~\cite{Conlonfox2020, Radziszowski2017}.

In 1975, Burr, Erd\H{o}s and Spencer \cite{Burr1975} investigated Ramsey numbers for disjoint union of small graphs.
Given a graph $G$ and a positive integer $n$, let $nG$ denote $n$ vertex-disjoint copies of $G$.
It was shown in \cite{Burr1975} that $r(nK_3)=5n$ for $n\ge 2$. A  \emph{book} $B_n$ is the union of $n$ distinct triangles having exactly one edge in common. In 1978, Rousseau and Sheehan \cite{RS1978} showed that the Ramsey number $r(B_n)\le 4n + 2$ for all $n$ and the bound is tight for infinitely many values of $n$ (\emph{e.g.}, when $4n+1$ is a prime power). A more general \emph{book} $B_n^{(k)}$ is the union of  $n$ distinct copies of complete graphs $K_{k+1}$, all sharing a common $K_k$ (thus $B_n = B_n^{(2)}$). Conlon \cite{Conlonadv2019} recently proved that for every $k$, $r(B_n^{(k)})=2^kn+o_k(n)$, answering a question of Erd\H{o}s, Faudree, Rousseau, and Schelp \cite{EFRS1978} and asymptotically confirming a conjecture of Thomason \cite{THOMASON1982}. More recently, Conlon, Fox, and Wigderson \cite{Conlonfoxw2020} provided another proof of Conlon's result.

Inspired by these old and recent results on $r(nK_3)$ and $r(B^{(k)}_n)$, in this paper we study the Ramsey number of fans.
A \emph{fan}  $F_n$ is a union of $n$ triangles sharing exactly one common vertex, named the \emph{center}, and all other vertices are distinct. Therefore, $nK_3$, $F_n$ and $B_n$ are three graphs formed by $n$ triangles that share zero, one, and two common vertices, respectively. Since $nK_3$ has more vertices than $F_n$ and $F_n$ has more vertices and edges than $B_n$, it is reasonable to believe that $r(B_n)\le r(F_n)\le r(n K_3)$ for sufficiently large $n$. We obtain the following bounds for $r(F_n)$ confirming $r(B_n) < r(F_n)$  for sufficiently large $n$.\footnote{These inequalities fail when $n=2$ because $r(B_2)= r(2K_3) = 10$ \cite{Burr1975, RS1978} while $r(F_2)=9$ \cite{LinLi2009}.} 

\begin{theorem}\label{maintheorem}
For every positive integer $n$,
\[{9n}/{2}  -5\le r(F_n)\le {11n}/{2}+6.\]
\end{theorem}
Theorem~\ref{maintheorem} improves previously best known bounds
\begin{align}\label{eq1}
 4n + 2 \le r(F_n) \le 6n.
 \end{align}
Indeed, Li and Rousseau \cite{LiR21996} first studied off-diagonal Ramsey numbers of fans. They showed that $r(F_1,F_n)=4n+1$ for $n\ge 2$ and $4n+1\le r(F_m,F_n)\le 4n+4m-2$ for $n\ge m\ge 1$. Lin and Li \cite{LinLi2009} proved that $r(F_2,F_n)=4n+1$ for $n\ge 2$ and improved the general upper bound as
\begin{align}\label{eq2}
r(F_m,F_n)\le 4n+2m  \quad \mbox{for}\quad n\ge m\ge 2.
 \end{align}
Lin, Li and Dong \cite{LIN20101228} showed that $r(F_m,F_n)=4n+1$ if $n$ is  sufficiently larger than $m$. The latest result for $r(F_m,F_n)$ is due to Zhang, Broersma and Chen \cite{ZBC2015}, who proved that
$r(F_m,F_n)=4n+1$ if $n\ge \max\{(m^2-m)/{2},{11m}/{2}-4\}$.
They also showed that $r(F_n, F_m)\ge 4n + 2$ for $m\le n < (m^2-m)/2$. This and \eqref{eq2} together give \eqref{eq1}.


The lower bound given in Theorem~\ref{maintheorem} is obtained from constructing a regular 3-partite graph with about $3n/2$ vertices in each part such that every vertex has less than $n$ neighbors in one of the other parts. To prove the upper bound given in Theorem~\ref{maintheorem}, we first find a large monochromatic clique in any 2-edge-colored $K_{11n/2 + 6}$ and then use this clique to find the desired copy of $F_n$. This approach is summarized in the following two lemmas.

\begin{lemma}\label{maintheorem2}
Let $m, n, N$ be positive integers such that $N = 4n+m+\left\lfloor\frac{6n}{m}\right\rfloor+1$.
Then every $2$-coloring of $E(K_N)$ yields a monochromatic copy of $F_n$ or $K_m$.
\end{lemma}

\begin{lemma}\label{fan}
Let $n$ be a positive integer. If a graph $G$  contains  a clique $V_0$ with $|V_0|\ge  3n/2+1$ such that every vertex $v\in V_0$ has at least $n$ neighbors in $V\backslash V_0$, then $G$ or its complement $\overline{G}$ contains a copy of $F_n$ with center in $V_0$.
\end{lemma}

We prove Lemmas~\ref{maintheorem2} and \ref{fan} by using the theorems of Hall and Tutte on matchings along with a result on $r(nK_2, F_m)$ from \cite{LinLi2009}.
Unfortunately our approach (of finding a large monochromatic clique) cannot prove $r(F_n)< 11n/2$  because Lemma~\ref{fan} is tight with respect to the size of $V_0$, see Section~\ref{sec:last} for details.

We organize our paper as follows. We give notation and preliminary results in Section 2. After proving Lemmas~\ref{maintheorem2} and \ref{fan} in Section 3, we complete the proof of Theorem~\ref{maintheorem} in Section 4. We give concluding remarks, including a lower bound for $r(F_n, F_m)$, in the last section.


%
%

\section{Notation and preliminaries}

We start this section with some notation and terminologies.  Given a positive integer $n$, let $[n] := \{1,2,\ldots,n\}$.
All graphs considered are simple and finite. Given a graph $G$, we denote by $V(G)$ and $E(G)$ the vertex and edge sets of $G$, respectively. $|G|:=|V(G)|$ and $|E(G)|$ are the \emph{order} and the \emph{size} of $G$, respectively. Let $\overline{G}$ denote the complement graph of $G$.

Given a graph $G$, let $v$ be a vertex and $H$ be a subgraph.  Denote by $N_H(v)$  the set of neighbors of $v$ in $H$.
For  a subset $S\subseteq V(G)$, define $N_H(S)=\cup_{v\in S}N_H(v)$.  The \emph{degree} of $v$ in $H$ is denoted by $d_H(v)$, that is, $d_H(v)=|N_H(v)|$. When all the vertices of $G$ have the same degree $d$, we call $G$  a \emph{$d$-regular graph}. The subgraph induced by the vertices of $S$ is denoted by $G[S]$. We simply write $G[V(G)\backslash S]$ as $G-S$. 
A component of $G$ is \emph{odd} if it consists of an odd number of vertices. We denote by $o(G)$ the number of odd components of $G$.


Given a graph $G$, we denote by $\nu(G)$ the size of a largest matching of $G$. We will use the following defect versions of Hall's and Tutte's theorems (see, \emph{e.g.}, \cite{Diestel2017}).

\begin{theorem}[Hall] \label{hall}
Let $G$ be a bipartite graph on parts $X$ and $Y$. For any non-negative integer  $d$,
$\nu(G) \ge |X|-d$ if and only if $|N_G(S)|\ge |S|-d$
 for every $S\subseteq X$.
\end{theorem}

\begin{theorem} [Tutte] \label{matching}
Let $G$ be a graph on order $n$.  For any non-negative integer  $d$,  $\nu(G) \ge (n-d)/{2}$ if and only if $o(G-S)\le |S|+d$ for every subset $S$ of $V(G)$.
\end{theorem}

The aforementioned result $r(F_n, F_m)\le 4n + 2m$ for $n\ge m$ follows from the following lemma, in which $n K_2$ is a matching of size $n$.
Note that the $n=m$ case of this lemma was proved in the same way as our Lemma~\ref{maintheorem2}.

\begin{lemma}[Lin and Li~\cite{LinLi2009}]
\label{ramsey}
Let $m,n$ be two positive integers with $n\ge m$. Then $r(nK_2, F_m)=2n+m$.
\end{lemma}

We will use the following corollary.

\begin{corollary}\label{3n}
Let $G$ be a graph with maximum degree $\Delta(G)$.  If  $\Delta(G) \ge 3n$, then $G$ or $\overline{G}$ contains
a copy of $F_n$.
\end{corollary}

\pf
Assume $v$ is a vertex such that $d_G(v)\ge 3n$.  By Lemma \ref{ramsey}, there is a copy of $nK_2$ in $G[N_G(v)]$ or a copy of $F_n$ in $\overline{G}[N_G(v)]$. So, $G$ has a copy of $F_n$ centered at $v$ or $\overline{G}$ contains a copy of $F_n$. \qed

\section{Proofs of Lemmas \ref{maintheorem2} and \ref{fan}}

\noindent {\bf Proof of Lemma \ref{maintheorem2}.}
Let  $c:=\left\lfloor\frac{6n}{m}\right\rfloor+1$ for convenience, and so $N=4n+m+c$.  Fix a red-blue edge coloring of $K_N$ and let
$R,B$ be the graphs induced by red and blue edges, respectively.
Assuming  there is no monochromatic $K_m$, we will find a monochromatic $F_n$.

Fix a vertex $w$. Assume, without loss of generality, that  $d_B(w)\ge \frac{N-1}{2}= 2n+\frac{m+c-1}{2}$.  Let $G:=B[N_B(w)]$.  If $\nu(G) \ge n$, we get a blue $F_n$ with center $w$. So, we assume $\nu(G) \le n-1$.  Applying Theorem \ref{matching} with $d:= d_B(w)-2n\ge \frac{m+c-1}{2}$, we get a subset $S\subseteq N_B(w)$ such that $o(G-S) \ge |S| + d  +1 \ge |S| + \frac{m+c+1}{2}$.

Let $C_1,C_2,\ldots,C_{\ell}$ be the vertex sets of  the components of $G-S$.  We have the following observations.
\begin{itemize}
\item [(a)] $\ell \ge o(G-S) \ge |S| + \frac{m+c+1}{2}$. 
\item [(b)] For any distinct $i,j\in [\ell]$, all edges between $C_i$ and $C_j$ are red.
\end{itemize}

 We further assume that $|C_1|:=\min\{|C_i|: i\in [\ell]\}$ and let $D=\cup_{i=2}^{\ell} C_i$.  By (b), $\overline{G}$ contains a red $K_{\ell}$, which in turn shows $\ell \le m -1$.

If  $d_B(w) \ge  3n$,  then by Corollary \ref{3n},  $N_B(w)$ spans  a blue $nK_2$ or a red $F_n$, which in turn shows that
there is a monochromatic $F_n$. So we assume $d_B(w) \le 3n -1$. By the minimality of $|C_1|$, we have the following.
\[
|C_1| \le  \frac{d_B(w) -|S|}{\ell} \le \frac{3n-1}{(m+c+1)/2}<\frac{3n}{m/2}=\frac{6n}{m}.
\]
Thus, $|C_1| \le \left\lfloor {6n}/{m}\right\rfloor$ and
\begin{align}
|D| & =    d_B(w) - |S| - |C_1| \nonumber \\
& \ge  2n+\frac{m+c-1}{2}- \left(\ell-\frac{m+c+1}{2}\right) -\left\lfloor\frac{6n}{m}\right\rfloor  \nonumber \\\
& = m+2n-\ell+1 \quad (\text{as } c=\left\lfloor {6n}/{m}\right\rfloor+1) \nonumber \\\
& \ge 2n +2. \label{eq:D}
\end{align}

For every $i\in [\ell]$, fix an arbitrary vertex $v_i\in C_i$. Let $X=\{v_2,v_3,\ldots,v_\ell\}$. Note that
$X\subseteq D$  and its vertices form a red clique,  and $v_1$ is red-adjacent to all vertices in $D$.

Let $D^*:= D\setminus X$. Then $|D^*| = |D| -(\ell-1) \ge  m +2n -2\ell +2$. We claim that $D^*$ contains a red matching of size at least  $n -\ell +2$. Otherwise, by removing the vertices of a largest red matching in $D^*$, we get
a blue clique $Z$  in $G[D^*]$  with   $|Z| \ge |D^*| - 2\nu(\overline{G}[D^*]) \ge m +2n -2\ell +2 - 2(n -\ell +1) = m$. So, $Z$ induces  a blue $K_{m}$, giving a contradiction. Let $M$ be a red matching in $\overline{G}[D^*]$ with $|M| \ge n -\ell +2$ and
let $Y := D^* - V(M)$.

Recall from (b) that $v_1$ is red-adjacent to all vertices in $D$. We will show that there is a red matching of size at least $n$ in $D$,  which gives a red $F_n$ with center $v_1$. Since $v_2$, $v_3$, $\dots$, $v_{\ell}$ are in different components of $G-S$, every vertex in $Y$ is red-adjacent to at least $|X|-1$ vertices in $X$. Hence we can greedily find a red matching $M'$ of size at least $\min\{|Y|,|X|-1\}$ between $X$ and $Y$. If $|M'|= |Y|$, then $M'\cup M$ saturates all the vertices in $D^*$.
Since $R[X]$ is a red complete graph, the vertices in $D= D^* \cup X$ contains a red matching of size at least $\left\lfloor {|D|} / {2}\right\rfloor\ge n$ by \eqref{eq:D}. If $|M'|\ge |X|-1$, then $|M'\cup M|\ge |X|-1+(n-\ell+2)= \ell -2 +(n -\ell +2) = n $.
In either case, we find a red matching of size at least $n$ in $D$, as desired. 
\qed
\\

\noindent {\bf Proof of Lemma \ref{fan}.}
Suppose to the contrary that neither $G$ nor $\overline{G}$ contains a  copy of $F_n$. We make the following observation:
\begin{align}\label{eq:V_0}
  \text{For every $v\in V_0$, there is no matching $M$ in $G[N(v)]$ such that $|V(M)\backslash V_0|\ge \left\lfloor \frac{n}{2}\right\rfloor$.}
\end{align}

Otherwise, there are $v\in V_0$ and a matching $M$ in $G[N(v)]$ such that $|V(M)\backslash V_0|\ge \left\lfloor{n}/{2}\right\rfloor$.
Since $V_0$ is a clique, $M$ can be extended to a matching $M^*$ containing all vertices in $V(M)\cup V_0\backslash\{v\}$ if $|V(M)\cup V_0\backslash\{v\}|$ is even and all but one vertex in $V(M)\cup V_0\backslash\{v\}$ if $|V(M)\cup V_0\backslash\{v\}|$ is odd.
Since $|V_0|\ge  \lceil 3n/2 \rceil +1$, it follows that $M^*$ is a matching $M$ in $G[N(v)]$ of size
 \[
 \left\lfloor\frac{|V(M)\cup V_0\backslash\{v\}|}{2}\right\rfloor \ge  \left\lfloor\frac{\left\lfloor{n}/{2}\right\rfloor + \left\lceil{3n}/{2}\right\rceil}{2}\right\rfloor= n,
\]
 which in turn gives an $F_n$ centered at $v$,  a contradiction.

In the rest of the proof, we will find disjoint subsets $S_{v_1}, S_{v_2}, \ldots,S_{v_{t}}$ of $V\backslash V_0$ for some $t>3$ and a vertex $w\in V_0$ such that  $\overline{G}[\cup_{1\le i\le t}S_{v_i}\cup\{w\}]$ contains a subgraph isomorphic to $F_n$.
For this goal, we first prove the following claim.

\begin{claim}\label{claim1}
For every vertex $v\in V_0$, there exists an independent set $S_v\subseteq N(v)\backslash V_0$ such that $|S_v|\ge |N(S_v)\cap V_0|+{n}/{2}$ and $|N(S_v)\cap V_0|\le {n}/{2}$.
\end{claim}

\pf
Let $v$ be a vertex in $V_0$ and $M_v$ be a largest matching in $G[N(v)\setminus V_0]$. Let $m:=|M_v|$.
Then $N(v)\setminus (V_0 \cup V(M_v))$ is an independent set.
Since $v$ has at least $n$ neighbors in $V\backslash V_0$, we have $| N(v)\setminus (V_0 \cup V(M_v)) | \ge n - 2m$.
Let $Z_v\subseteq N(v)\setminus (V_0 \cup V(M_v)) $ with $|Z_v|=n- 2m$.
If there is a matching $M'$ between $Z_v$ and $V_0\backslash\{v\}$ with $|M'| \ge \left\lfloor{n}/{2}\right\rfloor- 2m$, then $M := M'\cup M_v$ is a matching
with $|V(M)\backslash V_0|\ge \left\lfloor{n}/{2}\right\rfloor$, contradicting \eqref{eq:V_0}. Thus there is no matching of size $\left\lfloor{n}/{2}\right\rfloor - 2m =|Z_v|-\left\lceil{n}/{2}\right\rceil$ between $Z_v$ and $V_0\backslash\{v\}$. Applying Theorem \ref{hall} on $G\left[Z_v, V_0\backslash\{v\}\right]$ by taking
\[
X:=Z_v,\quad Y:=V_0\backslash\{v\} \quad\mbox{and}\quad d:=\left\lceil{n}/{2}\right\rceil,
\]
we get a subset $S_v\subseteq Z_v$ (thus $S_v$ is independent) such that
\[
|N(S_v)\cap V_0\backslash\{v\}|\le |S_v|-d-1.
\]
This implies that $|S_v|\ge |N(S_v)\cap V_0\backslash\{v\}|+1 + d\ge |N(S_v)\cap V_0|+{n}/{2}$ and
\[|N(S_v)\cap V_0|= |N(S_v)\cap V_0\backslash\{v\}|+1\le |S_v|-d \le |Z_v| - d\le {n}/{2}.
\]
This proves the claim.
\qed
\\

For every $v\in V_0$, let $S_v$ be the subset of $N(v)\backslash V_0$  defined in Claim~\ref{claim1}.
\begin{itemize}
\item Let  $v_1\in V_0$ such that  $|N(S_{v_1})\cap V_0|$ is the maximum among all vertices in $V_0$. Let $V_1:=V_0\backslash N(S_{v_1})$. By definition, every vertex in $V_1$ is not adjacent to any vertex in $S_{v_1}$.


\item For each $i \ge 1$, if $V_{i-1}\backslash N(S_{v_i}) \ne \emptyset$,  then define $V_i:=V_{i-1}\backslash N(S_{v_i})$ and choose $v_{i+1}\in V_i$ such that $|N(S_{v_{i+1}})\cap V_i|$ is the maximum among all vertices in $V_i$.
Note that $N(S_{v_{i+1}})\cap V_i\ne \emptyset$ because $v_{i+1}\in N(S_{v_{i+1}})\cap V_i$.  Together with the choice of $v_i$, we derive that
\begin{align}\label{eq:N'}
   0< |N(S_{v_{i+1}})\cap V_i|\le |N(S_{v_{i+1}})\cap V_{i-1}|\le |N(S_{v_{i}})\cap V_{i-1}|.
\end{align}
\end{itemize}
For simplicity, let $N'(S_{v_{i+1}}):=N(S_{v_{i+1}})\cap V_i$. By definition, $N'(S_{v_{1}})$, $N'(S_{v_{2}})$, $\dots$ are nonempty and pairwise disjoint.
Suppose the above process stops when $i= t$  due to $V_{t-1}\backslash N(S_{v_{t}}) = \emptyset$.
Then
\begin{align} \label{eq:UN}
\bigcup_{1\le i\le t}N'(S_{v_i})= V_0 \quad  \text{and} \quad \bigcup_{1\le i< t} N'(S_{v_i}) \subsetneq V_0.
\end{align}
By Claim \ref{claim1}, \eqref{eq:N'}, and the choice of $v_i$, we have
\begin{itemize}
  \item [(i)] $  |N'(S_{v_{t}})| \le  |N'(S_{v_{t-1}})| \le \cdots \le  |N'(S_{v_1}) | \le {n}/{2}$;
  \item [(ii)] $S_{v_1}, S_{v_2}, \dots, S_{v_t}$ are disjoint independent sets such that $|S_{v_i}|\ge |N'(S_{v_i})|+{n}/{2}$ for all $i\in [t]$;
  \item [(iii)] every vertex in $V_i$ is not adjacent to any vertex in $\bigcup_{1\le j\le i}S_{v_j}$ for all $i\in [t]$.
\end{itemize}
By \eqref{eq:UN} and (i), we have
\[
\frac{\sum_{i=1}^{t-1} | N'(S_{v_i})| }{t-1}\ge \frac{\sum_{i=1}^{t} |N'(S_{v_i})|}{t} = \frac{|V_0|}t
\quad \text{and}\quad
t\ge \frac{|V_0|}{|N'(S_{v_1})|} >  \frac{3n/2}{ n/2}= 3.
\]
It follows that
\begin{align*}
\sum_{i=1}^{t-1} |N'(S_{v_i})| & \ge |V_0|\cdot \frac{t-1}{t} \ge \frac{3n}2 \cdot \frac 23 = n.
\end{align*}
By (ii) and the fact that $t\ge 3$, we have
\[
\sum_{i=1}^{t-1} |S_{v_i} | \ge \sum_{i=1}^{t-1} \left( |N'(S_{v_i})| + \frac n2 \right) \ge  n + \frac{n}2\cdot 2 = 2n.
\]

Since all $S_{v_i}$ are independent sets, we obtain a matching $M'$ of size $n$ in $\overline{G} \left[\bigcup_{i=1}^{t-1}S_{v_i} \right]$.
Since $\bigcup_{i=1}^{t-1}N(S_{v_i}) \subsetneq V_0$, there is a vertex $w\in V_0\setminus \bigcup_{i=1}^{t-1}N(S_{v_i})$. By (iii), $w$ is not adjacent to any vertex in $\bigcup_{i=1}^{t-1} S_{v_i}$. Therefore, $V(M')\cup\{w\}$ spans a fan $F_n$ in $\overline{G}$.
\qed

\section{Proof of Theorem \ref{maintheorem}}



\subsection{Lower bound}\label{lowbound}

Let $n$ be a positive integer and let $t$ be  the largest even number less than ${3n}/{2}$. Thus $t\ge 3n/2 - 2$.
We construct a graph $G=(V,E)$ on $3t$ vertices as follows.  Let $V_1\cup V_2\cup V_3$  be a partition of  $V$ such that $|V_1| = |V_2| = |V_3| = t$ and all $G[V_i]$ are complete graphs.
For each $i \in [3]$,  further partition $V_i$ into two subsets  $X_i$ and $Y_i$ with $|X_i| = |Y_i| = t/2$,  and add edges between $X_i$ and $Y_{i+1}$ such that $G[X_i, Y_{i+1}]$ is an $\left\lceil\frac{n}{2}\right\rceil$-regular bipartite graph, where we assume  $Y_4=Y_1$.  The graph $G$ is depicted in Figure~\ref{fig1}.

Observe that $G$ does not contain a copy of $F_n$ because every vertex has degree $\left\lceil{n}/{2}\right\rceil+t-1< 2n$.
To see that   $\overline{G}$ contains no  copy of $F_n$, we note that $\overline{G}$ is 3-partite because $V_1, V_2, V_3$ induce cliques in $G$. Thus $\overline{G}$ induces a bipartite graph on $N_{\overline{G}}(v)$ for every vertex $v\in V$.  Furthermore, two parts of this bipartite graph have sizes $t$ and $t- \left\lceil{n}/{2}\right\rceil< n$ and thus there is no matching of size $n$ in $\overline{G}[N_{\overline{G}}(v)]$. Consequently $\overline{G}$ contains no copy of $F_n$.

Since neither $G$ nor $\overline{G}$ contains a copy of $F_n$, we have $r(F_n)\ge |V|+1= 3t+1\ge {9n}/{2}-5$.

%
%
%
%
%
%
%
%

\begin{figure}[htbp]
 \begin{center}
\scalebox{0.07}[0.07]{\includegraphics {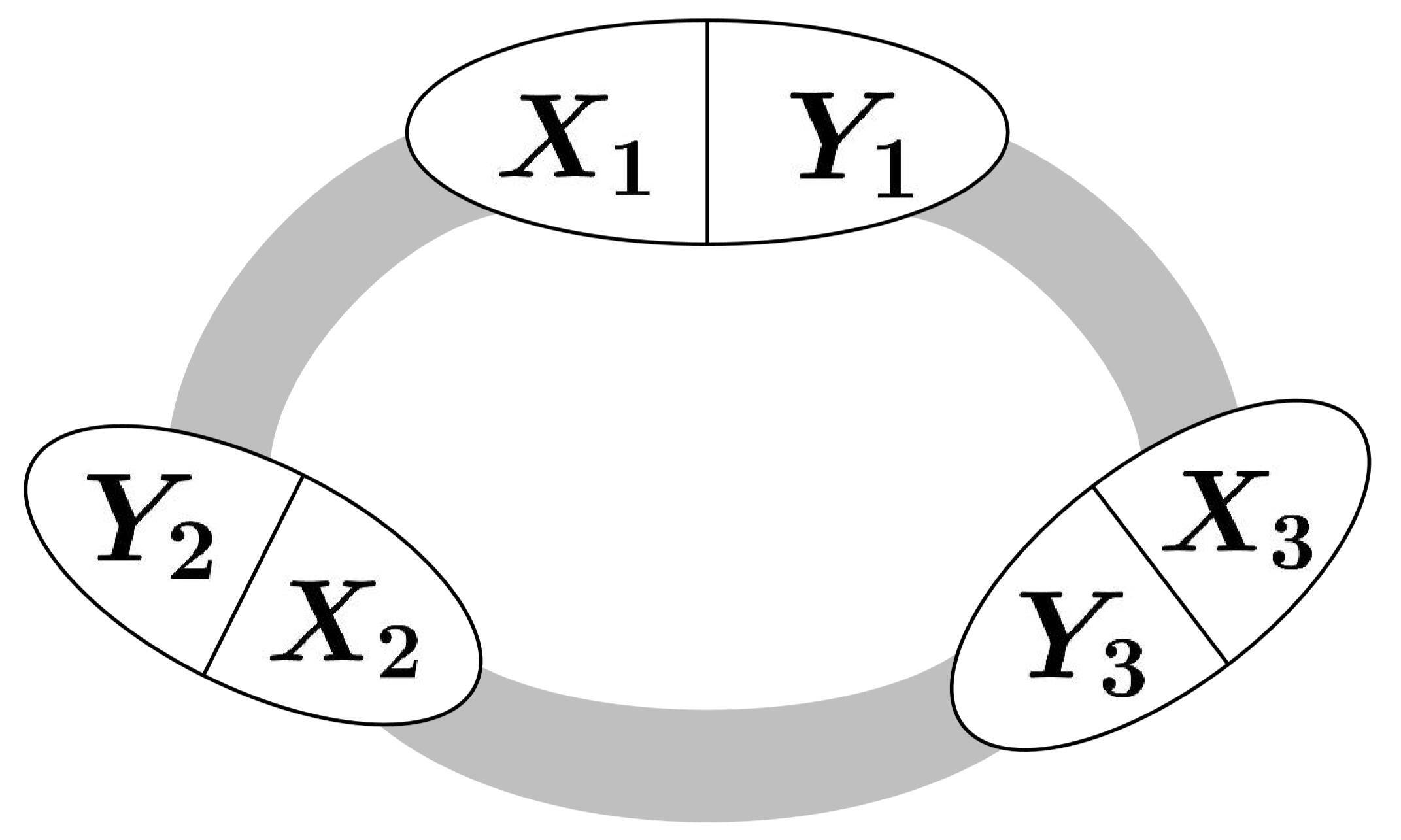}}
 \end{center}
 \caption{Illustration of $G$} \label{fig1}
\end{figure}

\subsection{Upper bound}

Given a red-blue edge coloring of a complete graph on $N=\left\lceil{11n}/{2}\right\rceil+5$, let $R,B$ be the graphs induced by the red and blue edges, respectively.
If there is a vertex $v$ with $|N_R(v)|\ge 3n$ or $|N_B(v)|\ge 3n$, then there is a monochromatic $F_n$ by Corollary \ref{3n}. We thus assume that $|N_R(v)|\le 3n-1$ and $|N_B(v)|\le 3n-1$ for all vertices $v$. Because $R$ and $B$ are complementary to each other, it follows that $d_R(v),d_B(v)\ge (N-1)-(3n-1)=N-3n$.  Define $m:=N-4n-4=\left\lceil{3n}/{2}\right\rceil+1$. Since
\[
\frac{6n}{m} =\frac{6n}{\left\lceil{3n}/{2}\right\rceil+1}< \frac{6n}{{3n}/{2}}=4,
\]
we have $\left\lfloor\frac{6n}{m}\right\rfloor\le 3$. So $4n+m+\left\lfloor\frac{6n}{m}\right\rfloor+1\le N$.
By Lemma \ref{maintheorem2},
there exists a monochromatic $F_n$ or a monochromatic $K_m$ in $K_N$. If there exists a  monochromatic $F_n$, we are done. Otherwise, assume there is a monochromatic $K_m$. Without loss of generality, suppose that $K_m$ is blue. Let $V_0$ be the blue clique of order $m$. For every $v\in V_0$, $v$ has at least $d_B(v)-(m-1)\ge (N-3n)-(N-4n-5) =n+5>n$ neighbors in $V(B)\backslash V_0$. Applying Lemma \ref{fan} with $G:=B$, we get a monochromatic $F_n$. Thus  $r(F_n)\le N\le {11n}/{2}+6$.
\qed

\section{Concluding remarks}
\label{sec:last}

Theorem \ref{maintheorem} contains upper and lower bounds for $r(F_n)$ that differ by about $n$. We do not have a conjecture on the value of $r(F_n)$ but speculate that the lower bound is closer to the truth.

As mentioned in Section~1, we believe that $r(F_n)\le r(n K_3)=5n$. Although we are unable to verify this, there is some evidence for this assertion. First, $r(F_2)=9< 10= r(2 K_3)$. Second, let $t, n$ be positive integers such that $t$ divides $n$. One way of proving $r(F_n)\le r(n K_3)$ is showing that $r( \frac{n}t F_t)\le r(n K_3)$ for all such $t$. Indeed, Burr, Erd\H{o}s and Spencer \cite{Burr1975} proved the following theorem.
\begin{theorem}\emph{(\cite[Theorem 1]{Burr1975})} \label{thm:BES}
Let $n$ be a positive integer and $G$ be a graph of order $k$ and independence number $i$.
Then there exists a constant $C= C_G$ such that
\[
(2k-i) n - 1 \le r(n G)\le (2k-i) n + C.
\]
\end{theorem}
We can apply Theorem~\ref{thm:BES} with $G= F_t$ (thus $k= 2t+1$ and $i=t$) and obtain that $ (3t + 2) \frac{n}t - 1 \le r( \frac{n}t F_t) \le (3t + 2) \frac{n}t + C$ for some $C$ depending only on $F_t$. For fixed $t\ge 2$, this implies that  $r( \frac{n}t F_t)= \left(3 + \frac{2}t \right) n + O(1)$, much smaller than $r(n K_3)$.\footnote{The proof of \cite[Theorem 1]{Burr1975} shows that $C$ is double exponential in $t$ and thus $r( \frac{n}t F_t)=  \left(3 + \frac{2}t \right) n+ o(n)$ whenever $t=o(\log\log n)$.}

\medskip
We now give a construction that shows Lemma \ref{fan} is best possible with respect to $|V_0|$. Suppose $n$ is even.
Let $G=(V, E)$ be a graph on $9n/2-2$ vertices that contains a clique $V_0$ of order $3n/2$, and $V_0$ is partitioned into $V_1\cup V_2\cup V_3$ such that $|V_1|= |V_2| = |V_3|= n/2$. The set  $V\backslash V_0$ is independent and is partitioned into $U_1\cup U_2\cup U_3\cup\{x_0\}$ with $|U_1|= |U_2| = |U_3|= n-1$. For every $i\in [3]$, $G[V_i, U_i]$ is complete but $G[V_i, U_j]$ is empty for distinct $i,j\in [3]$. In addition, all the vertices of $V_0$ are adjacent to $x_0$. Then each $v\in V_0$ has exactly $n$ neighbors in $V\setminus V_0$. But neither $G$ or $\overline{G}$ contains an $F_n$ centered in $V_0$ (there are copies of $F_n$ whose centers are outside $V_0$ in  $\overline{G}$). Indeed, for $v\in V_0$, every matching $M$ in $G[N_G(v)]$ contains at most $n/2$ vertices in $V\setminus V_0$ and thus $|V(M)|\le  |V_0| -1 + n/2 < 2n$. In $\overline{G}$, every $v\in V_0$ has exactly $2n-2$ neighbors so there is no matching of order $2n$ in $\overline{G}[N_{\overline{G}}(v)]$.
\medskip
We can generalize the construction that gives the lower bound of Theorem \ref{maintheorem} and obtain a new lower bound for $r(F_n,F_m)$. When $m\le n<{3m}/{2}-7$, our bound is better than $r(F_n, F_m)\ge 4n + 2$ given in \cite{ZBC2015}.
\begin{theorem}
Let $m,n$ be positive integers with $m\le n\le \frac{3m}{2}-3$. We have
\[
r(F_n,F_m)\ge \frac{3m}{2}+3n-5.
\]
\end{theorem}
\begin{proof}
We construct a graph $G=(V, E)$ on $3t$ vertices, where $t$ is the largest even number less than $\frac{m}{2}+n$. Thus $t\ge \frac{m}{2}+n-2$. Our goal is to show that neither $G$ contains $F_n$ nor $\overline{G}$ contains $F_m$. This will imply that $r(F_n,F_m)\ge 3t+1\ge 3m/2+3n-5$ as desired.

Let $V_1\cup V_2\cup V_3$ be a partition of $V$ such that  $|V_1|=|V_2|=|V_3|=t$ and all $G[V_i]$ are complete graphs. For every $i\in[3]$, partition $V_i$ into two subsets $X_i$ and $Y_i$ with $|X_i|=|Y_i|=t/2$. Observe that
\[
\frac{t}{2}-\left\lceil n-\frac{m}{2}\right\rceil\ge \frac{m}{4}+\frac{n}{2}-1-\left( n-\frac{m}{2} + \frac12 \right)=\frac{3m}{4}-\frac{n}{2}- \frac32\ge 0
\quad\mbox{as $n\le \frac{3m}{2}-3$}.
\]
For every $i\in[3]$, we add edges between $X_i$ and $Y_{i+1}$  (assuming $Y_4=Y_1$) such that $G[X_i, Y_{i+1}]$ is an $\left\lceil n-\frac{m}{2}\right\rceil$-regular bipartite graph.

The graph $G$ contains no $F_n$ because for every vertex $v\in V$,
\[
d_G(v)\le t-1+ \left\lceil n-\frac{m}{2}\right\rceil <\frac{m}{2}+n-1+n-\frac{m}{2} + \frac12 < 2n \quad \text{as } t<\frac{m}2+n.
\]
For every $v\in V$, $\overline{G}$ induces a bipartite graph on $N_{\overline{G}}(v)$ with one part of size
\[
t-\left\lceil n-\frac{m}{2}\right\rceil<\frac{m}{2}+n-\left(n-\frac{m}{2}\right)=m.
\]
It follows that $\overline{G}$ contains no $F_m$.
\end{proof}


\section*{Acknowledgement}
We thank anonymous referees for their comments that improved the presentation of this paper.

\end{document}